\documentclass[11pt,a4paper]{amsart}
\usepackage[T1]{fontenc}
\usepackage{lmodern,amsmath,amssymb,amsthm,mathtools}
\usepackage[margin=1in]{geometry}
\usepackage{microtype}
\usepackage[colorlinks=true,linkcolor=blue,citecolor=blue,urlcolor=blue]{hyperref}
\usepackage{enumitem}
\hypersetup{pdftitle={A nearly linear bound for the Lovász conjecture},pdfauthor={Bowen Li and Abhishek Methuku},pdfsubject={Long paths and cycles in connected vertex-transitive graphs}}
\setlist{itemsep=3pt,topsep=5pt}
\newtheorem{theorem}{Theorem}[section]
\newtheorem{lemma}[theorem]{Lemma}

\theoremstyle{definition}
\newtheorem{definition}[theorem]{Definition}
\theoremstyle{remark}
\newtheorem{remark}[theorem]{Remark}
\newcommand{\Cay}{\operatorname{Cay}}
\newcommand{\Aut}{\operatorname{Aut}}
\newcommand{\diam}{\operatorname{diam}}
\newcommand{\dist}{\operatorname{dist}}

\newcommand{\rank}{\operatorname{rank}}
\newcommand{\Z}{\mathbb Z}
\newcommand{\E}{\mathbb E}
\newcommand{\eps}{\varepsilon}
\newcommand{\normal}{\trianglelefteq}
\title[A nearly linear bound for the Lov\'asz conjecture]{A nearly linear bound for the Lov\'asz conjecture}
\author[Bowen Li]{Bowen Li$^{\star}$}
\address{$^{\star}$Department of Mathematics, University of Illinois Urbana--Champaign}
\email{bowenl6@illinois.edu}
\thanks{Bowen Li is supported by NSF grant RTG DMS-1937241, UIUC Campus Research Board RB24012, and University Block Grant Fellowship.}

\author[Abhishek Methuku]{Abhishek Methuku$^{\star}$}
\email{methuku@illinois.edu}
\thanks{Abhishek Methuku received funding from Simons award SFI-MPS-TSM-00025583 and the UIUC Campus Research Board Award RB25050.}
\date{18 September 2026}
\begin{document}
\begin{abstract}
The celebrated conjecture of Lov\'asz from 1969 asks whether every connected vertex-transitive graph has a Hamiltonian path. Buci\'c, Christoph, Pokrovskiy and Steiner recently proved that every such graph on \(n\) vertices contains a cycle of length \(n^{2/3-o(1)}\).

In this paper, we improve this bound to \(n^{1-o(1)}\). Our proof uses a structure theorem of Tessera and Tointon to first obtain a partition of the vertex set into sets of small diameter in the original graph. When the parts are large, we repeatedly traverse a spanning tree of maximum degree at most three in the quotient graph, and use the Lov\'asz local lemma to join random short paths along this traversal and extract a long path in the original graph. When the parts are small, we apply Babai's contraction lemma to reduce the problem to finding a long path in a connected Cayley graph of a nilpotent group with boundedly many generators and bounded nilpotency class, and then show that such a Cayley graph on \(m\) vertices contains a path on \(m^{1-o(1)}\) vertices.
\end{abstract}
\maketitle
\section{Introduction}\label{sec:intro}
A path or cycle in a graph is \emph{Hamiltonian} if it contains every vertex. Finding conditions that guarantee the existence of such paths and cycles is a classical problem in graph theory. Dirac's theorem~\cite{Dirac}, for example, states that every graph on \(n\ge3\) vertices with minimum degree at least \(n/2\) has a Hamiltonian cycle. A different possible condition is symmetry, which can also be present in graphs of small degree. An \emph{automorphism} of a graph is a permutation of its vertices such that two vertices are adjacent if and only if their images are adjacent. A graph \(X\), with vertex set \(V(X)\), is \emph{vertex-transitive} if, for any \(u,v\in V(X)\), some automorphism maps \(u\) to \(v\). Throughout, graphs are finite and undirected, with no loops or multiple edges unless explicitly stated otherwise.

In 1969, Lov\'asz~\cite{Lovasz} asked whether every connected vertex-transitive graph has a Hamiltonian path. This is now known as the Lov\'asz conjecture. A stronger conjecture, commonly attributed to Thomassen in 1978, asserts that, with finitely many exceptions, every connected vertex-transitive graph has a Hamiltonian cycle; see~\cite{Babai,NSTW}. The Petersen graph shows that exceptions are necessary for the cycle conjecture. These conjectures apply even to vertex-transitive graphs of degree three, for which conditions requiring large minimum degree give no information when the number of vertices is large.

For Cayley graphs, related questions go back to Rankin~\cite{Rankin} in 1948 and Rapaport-Strasser~\cite{Rapaport} in 1959; Rankin's work was motivated by change ringing. A Cayley graph has the elements of a finite group as vertices, with edges corresponding to multiplication by elements of a generating set closed under inverses. Every Cayley graph is vertex-transitive, and it is conjectured that every connected Cayley graph on at least three vertices has a Hamiltonian cycle. We refer to Kutnar and Maru\v si\v c~\cite{KM} for a survey of Hamiltonian paths and cycles in vertex-transitive graphs.

The conjectures have been proved in several important cases. Christofides, Hladk\'y and M\'ath\'e~\cite{CHM} showed that, for every fixed \(\alpha>0\), every sufficiently large connected vertex-transitive graph on \(n\) vertices with degree at least \(\alpha n\) has a Hamiltonian cycle. Bedert, Dragani\'c, M\"uyesser and Pavez-Sign\'e~\cite{BDMP} proved that there is an absolute constant \(\delta>0\) such that every sufficiently large connected Cayley graph on \(n\) vertices with degree greater than \(n^{1-\delta}\) has a Hamiltonian cycle. For random Cayley graphs, Dragani\'c, Montgomery, Munh\'a Correia, Pokrovskiy and Sudakov~\cite{DMMPS} proved the following: for a sufficiently large absolute constant \(C\), if \(G\) is a group of order \(n\) and \(S\) is chosen uniformly at random among its subsets of size \(d\ge C\log n\), then the Cayley graph generated by \(S\) and the inverses of its elements has a Hamiltonian cycle with probability tending to one as \(n\to\infty\).

Other results establish the conjectures for specific families of graphs. For integers \(k\ge1\) and \(m\ge2k+1\), the Kneser graph \(K(m,k)\) has the \(k\)-element subsets of an \(m\)-element set as vertices, with two vertices adjacent if the corresponding sets are disjoint. Merino, M\"utze and Namrata~\cite{MMN} proved that every such graph has a Hamiltonian cycle, except for the Petersen graph \(K(5,2)\). Brada\v c and Janzer~\cite{BJ} use expansion methods to obtain further results on the Hamiltonicity of Cayley and Kneser graphs, including versions that allow random edge deletions.

Using bounds on intersections of longest cycles, Ma and Zhao~\cite[Corollary~1.7]{MZ} proved that every connected \(d\)-regular vertex-transitive graph on \(n\) vertices, with \(n>d\ge2\), contains a cycle of length \(\Omega(d^{1/3}n^{1/2})\). More recently, Christoph, Hunter and Sudakov~\cite[Theorem~1.8]{CHS} proved that every connected vertex-transitive graph on \(n\) vertices with degree \(d\) satisfying \(d/(\log^3 n\,\log\log n)\to\infty\) as \(n\to\infty\) contains a cycle, and hence a path, with \(n-o(n)\) vertices. For arbitrary connected vertex-transitive graphs, much less is known.

For a nonempty graph \(X\), let \(p(X)\) be the largest number of vertices in a path in \(X\), and let \(c(X)\) be the maximum length of a cycle in \(X\), with \(c(X)=0\) when \(X\) has no cycle. For \(n\ge3\), define
\[
 f(n)=\min\{c(X): X\text{ is a connected vertex-transitive graph on }n\text{ vertices}\}.
\]
Thus Thomassen's conjecture on Hamiltonian cycles in connected vertex-transitive graphs asserts that \(f(n)=n\) for all sufficiently large \(n\).

In 1979, Babai~\cite{Babai} proved that \(f(n)=\Omega(\sqrt n)\). This remained the best general lower bound for more than forty years. DeVos~\cite{DeVos} improved it to \(\Omega(n^{3/5})\) in 2023, followed by the bound \(\Omega(n^{13/21})\) of Groenland, Longbrake, Steiner, Turcotte and Yepremyan~\cite{GLSTY}. Norin, Steiner, Thomass\'e and Wollan~\cite{NSTW} then obtained \(\Omega(n^{9/14})\). Most recently, Buci\'c, Christoph, Pokrovskiy and Steiner proved the following.

\begin{theorem}[Buci\'c, Christoph, Pokrovskiy and Steiner~\cite{BCPS}]\label{thm:BCPS}
For every \(\eps>0\), every sufficiently large connected vertex-transitive graph on \(n\) vertices contains a cycle of length at least \(n^{2/3-\eps}\).
\end{theorem}

Buci\'c et al. obtain their bound through the longest-path transversal approach, building on work of DeVos~\cite{DeVos}, Groenland et al.~\cite{GLSTY}, and Norin et al.~\cite{NSTW}, with the key step showing that the union of a cycle of length \(2k\) and a perfect matching between its two consecutive blocks of \(k\) vertices contains a path using \(k^{1-o(1)}\) matching edges. Their proof builds on the expander decomposition lemma of Letzter, Methuku and Sudakov~\cite[Lemma~4.1]{LMS} and other tools for expanders building on work of Buci\'c and Montgomery~\cite{BM}, which in turn uses ideas of Tomon~\cite{Tomon}.

In this paper, we exploit the structure of vertex-transitive graphs directly to break the \(n^{2/3}\) barrier, improving the lower bound for both paths and cycles from \(n^{2/3-o(1)}\) to a near-linear bound of \(n^{1-o(1)}\).

\begin{theorem}\label{thm:main}
For every \(\eps>0\), every sufficiently large connected vertex-transitive graph on \(n\) vertices contains a cycle of length at least \(n^{1-\eps}\).
\end{theorem}

Our proof uses the structure theorem of Tessera and Tointon~\cite{TT} to partition the vertex set into parts of small diameter in the original graph, in such a way that every automorphism permutes the parts. Let \(Q\) be the quotient graph whose vertices are these parts, with two distinct parts adjacent when an edge of the original graph joins them. The graph \(Q\) is vertex-transitive and has a spanning tree of maximum degree at most three.

When there is only one part, a standard bound relating diameter to longest-path length gives the required path (Lemma~\ref{lem:diameter}).

When the parts are large and there are at least two of them, we choose a walk consisting of repeated traversals of the spanning tree, each using every edge once in each direction. The degree bound in the spanning tree ensures that each traversal visits any part at most four times. We follow this walk using randomly chosen edges between parts and random short paths with both endpoints in the visited part. These short paths may leave their parts. Vertex-transitivity gives a uniform bound on the expected number of these paths through any vertex. We use this bound and the Lov\'asz local lemma to control intersections and extract a long simple path.

When the parts are small, the number of parts is large enough that a path of near-linear length in \(Q\) yields the desired bound. Indeed, every path in \(Q\) lifts to a path with the same number of vertices in the original graph. Using the additional group structure supplied by the Tessera--Tointon theorem, we apply Babai's contraction lemma~\cite{BabaiContraction} to reduce this task to finding a long path in a connected Cayley graph of a nilpotent group with boundedly many generators and bounded nilpotency class. We then show that such a Cayley graph on \(m\) vertices contains a path on \(m^{1-o(1)}\) vertices (Theorem~\ref{thm:nilpotent}). Since \(m\) is at least a fixed positive proportion of \(|V(Q)|\), this yields the desired long path in \(Q\).

\subsection{Structure of the paper}
Section~\ref{sec:overview} gives an overview of the proof, and Section~\ref{sec:prelim} collects notation and preliminaries. Section~\ref{sec:routes} constructs a long path when the orbits of a normal subgroup are large, by repeatedly traversing a spanning tree in the quotient and using the Lov\'asz local lemma to join random short paths along this traversal. In Section~\ref{sec:cyclic}, we prove a quantitative variant of the Factor Group Lemma~\cite[Section~2.2]{WitteGallian} for quotients by cyclic subgroups of the centre, with the loss in the proportion of vertices covered bounded in terms of the word length of a generator of the subgroup. Section~\ref{sec:nilpotent} applies this construction by induction on nilpotency class to prove the near-linear bound for nilpotent Cayley graphs. Finally, Section~\ref{sec:assembly} applies the structure theorem of Tessera and Tointon~\cite{TT}, uses Babai's contraction lemma for the small-orbit case, and combines the bounds to prove Theorem~\ref{thm:main}.

\medskip
\noindent\textbf{Note added.}
After the first version of this paper was posted on arXiv, we learned that
Buci\'c, Christoph, Pokrovskiy and Steiner~\cite{BucicPalomar} had
independently proved Theorem~\ref{thm:main} (a working draft of their proof
is available in~\cite{BCPSNearLinear}), and that Ma and
Zhao~\cite{MZLinear} had proved the linear bound.

All three proofs follow different strategies. Our proof combines the
Tessera--Tointon structure theorem with the Lov\'asz local lemma for large
orbits and a reduction via Babai's contraction lemma to nilpotent Cayley
graphs for small orbits. Buci\'c, Christoph, Pokrovskiy and Steiner also use
the Tessera--Tointon theorem, but they combine it with their earlier
cycle-and-matching theorem and group-theoretic arguments for lifting
paths. Ma and Zhao combine a cycle-merging argument adapted from Segal
and Verstra\"ete with the structure theorem of DeVos and Mohar.

\section{Proof overview}\label{sec:overview}
Fix \(0<\eps<1\), and let \(X\) be a connected vertex-transitive graph on \(n\) vertices, where \(n\) is sufficiently large in terms of \(\eps\). Our aim is to find a path with at least \(n^{1-\eps}\) vertices. We begin by using the finite form of the Tessera--Tointon structure theorem stated in Theorem~\ref{thm:TT} to partition the vertex set into sets of diameter at most \(n^{\eps/8}\), with distances measured in \(X\).

The parts are the orbits of a normal subgroup \(H\) of the automorphism group of \(X\): the orbit of a vertex \(v\) is \(Hv=\{hv:h\in H\}\). Here normality means that \(gHg^{-1}=H\) for every automorphism \(g\) of \(X\). Every automorphism therefore permutes the orbits, and vertex-transitivity makes their sizes equal. Write \(s\) for this common size. The quotient graph \(Q=X/H\) has one vertex for each orbit, with two distinct orbits adjacent when an edge of \(X\) joins them. It is again connected and vertex-transitive. If \(q\) is its number of vertices, then \(n=qs\).

We first consider the case of a single orbit, where a diameter bound gives the required long path. With multiple orbits, small diameter does not guarantee that an orbit induces a connected subgraph: short paths between its vertices may leave the orbit. When the orbits are large, we follow repeated traversals of a spanning tree of the quotient, using random short paths and controlling their intersections. When they are small, the quotient has many vertices, and we use the additional group structure supplied by the Tessera--Tointon theorem to find a long path in the quotient.

\subsection{A single orbit: the diameter bound}\label{sec:overviewsingle}
If \(q=1\), the diameter \(D\) of \(X\) is at most \(n^{\eps/8}\). Babai and Szegedy~\cite{BabaiSzegedy} show that, for every \(A\subseteq V(X)\) with \(1\le |A|\le n/2\), at least \(|A|/(D+1)\) vertices outside \(A\) have a neighbour in \(A\). We combine this expansion estimate with a standard result on finding long paths in expanders~\cite[Proposition~7.1]{Krivelevich} to obtain a path containing at least
\[
 \frac{n}{2(D+1)}\ge\frac{n^{1-\eps/8}}{4}\ge n^{1-\eps}
\]
vertices for sufficiently large \(n\); see the proof of Lemma~\ref{lem:diameter} for details.

\subsection{Large orbits: joining random paths}\label{sec:overviewroutes}
Suppose \(q\ge2\) and \(s\ge n^{\eps/2}\). We seek a long path by visiting each orbit many times, using a short path between two of its vertices at each visit. Theorems of Mader~\cite{Mader} and Win~\cite{Win} give a spanning tree of \(Q\) of maximum degree at most three (see Lemma~\ref{lem:tree}). Choose a closed walk in this tree that traverses every edge exactly once in each direction. Repeat this closed walk \(t\) times, where the positive integer \(t\) will be chosen below. The resulting walk through \(Q\) visits each orbit at most \(3t+1\) times.

For each step of this walk in \(Q\), independently choose an edge of \(X\) uniformly at random between the two corresponding orbits. At each intermediate visit, join the endpoint of the incoming edge in the visited orbit to the endpoint of the outgoing edge in that orbit. At the first and last visits, choose an initial and a terminal vertex uniformly at random from the respective orbits, independently of one another and of all edges, and join them to the endpoints in those orbits of the first and last chosen edges, respectively.

Each chosen edge has a uniformly distributed endpoint in each of its two orbits. Indeed, this is because \(H\) preserves both orbits and acts transitively on each, so every vertex in one orbit has the same number of neighbours in the other. At an intermediate visit, the two endpoints to be joined come from independent choices of the incoming and outgoing edges. They are therefore independent, even when these two steps traverse the same pair of orbits. The independent choices of the initial and terminal vertices give the same conclusion at the first and last visits.

Once the connecting edges and the initial and terminal vertices have been chosen, the two endpoints at every visit are fixed. At each visit, choose a path uniformly at random from all shortest paths in \(X\) between its two endpoints. Make these path choices independently for different visits. Since both endpoints lie in an orbit of diameter at most \(n^{\eps/8}\) in \(X\), each chosen path has at most \(n^{\eps/8}\) edges. We write \(R=n^{\eps/8}+1\) for the resulting bound on its number of vertices. Paths from nonconsecutive visits depend on disjoint random choices and are therefore independent.

The main difficulty is that these paths may intersect, even when their endpoints lie in different orbits. We use vertex-transitivity to bound how many of the random paths are expected to contain a given vertex. If we choose one such random path for each orbit, this expectation is the same for every vertex of \(X\). There are \(q\) paths, each with at most \(R\) vertices, so the expectation is at most \(qR/n=R/s\). Since each orbit occurs at most \(3t+1\) times in our repeated walk, the expected number of chosen paths containing any given vertex is at most \((3t+1)R/s\). By independence, summing this bound over the at most \(R\) vertices of the path at a fixed visit shows that the sum of its intersection probabilities with paths from nonconsecutive visits is at most \((3t+1)R^2/s=O(tR^2/s)\).

Each short path depends only on its own path choice and the adjacent connecting-edge choices, with the choice of the initial or terminal vertex replacing the missing edge choice at the first or last visit. The Lov\'asz local lemma therefore allows us to take \(t\) to be a sufficiently small constant multiple of \(s/R^2\) and avoid all intersections between paths from nonconsecutive visits. Label each vertex by the first visit whose short path contains it. In the subgraph consisting of these paths and the chosen connecting edges, an edge changes this label by at most two. A path from the initial to the terminal vertex therefore has at least \(t(q-1)\) vertices; see the proof of Lemma~\ref{lem:largefibres}. This yields a path with
\[
 \Omega(qt)=\Omega(n/R^2)=\Omega(n^{1-\eps/4})
\]
vertices, which exceeds \(n^{1-\eps}\) for sufficiently large \(n\).

\subsection{Small orbits: a long quotient path}\label{sec:overviewactions}
Suppose \(q\ge2\) and \(s<n^{\eps/2}\), so \(q>n^{1-\eps/2}\). Since every path in \(Q=X/H\) lifts to a path in \(X\) with the same number of vertices, it suffices to find a path in \(Q\) with \(q^{1-o(1)}\) vertices.

The Tessera--Tointon theorem supplies a nilpotent group \(N\le\Aut(Q)\) acting freely with boundedly many orbits, with both its number of generators and its nilpotency class bounded in terms of \(\eps\). Babai's contraction lemma~\cite{BabaiContraction} then gives a connected Cayley graph \(Y\) of \(N\) by contracting connected parts of \(Q\). This graph has \(m=|N|\) vertices, and every path in it expands to a path in \(Q\) with at least as many vertices.

Before explaining how to find a path in \(Y\) with \(m^{1-o(1)}\) vertices (Theorem~\ref{thm:nilpotent}), let us first show how this gives the desired bound for \(X\). Since \(N\) acts freely on \(Q\) with boundedly many orbits, \(m\) is at least a fixed positive fraction of \(q\). Expanding such a path in \(Q\) and then lifting it to \(X\) therefore gives a path with \(q^{1-o(1)}\ge n^{1-\eps}\) vertices for sufficiently large \(n\).

Finally, Watkins's theorem~\cite{Watkins} ensures that \(X\) is 3-connected unless it is already a cycle. Bondy and Locke~\cite{BL} then give a cycle with at least two fifths as many edges as a longest path. Thus the near-linear path bound also gives a cycle of length \(n^{1-o(1)}\).

\vspace{8pt}
\noindent\emph{Long cycles in nilpotent Cayley graphs.}
We now explain why every connected Cayley graph of a nilpotent group \(N\) of order \(m\ge3\), with boundedly many generators and bounded nilpotency class, contains a cycle, and hence a path, with \(m^{1-o(1)}\) vertices. We form a sequence of quotient groups, at each nonabelian stage quotienting by the last nontrivial term of the lower central series. This subgroup lies in the centre, and the quotient has smaller nilpotency class, so the sequence ends at an abelian group. We then further refine each step into successive quotients by cyclic subgroups of the centre. The Cayley graph of the final abelian quotient has a Hamiltonian cycle~\cite{Marusic}. Working backwards through the sequence, we lift this cycle to a long cycle in the original graph, controlling the loss in the proportion of vertices covered at each step.

Our lifting argument builds on the classical Factor Group Lemma~\cite[Section~2.2]{WitteGallian}: if the generator labels around a quotient cycle multiply to a generator of the cyclic subgroup, repeating the cycle lifts it to a cycle covering the same proportion of vertices in the original graph. A long quotient cycle need not satisfy this condition. We instead use a short word representing a generator of the central cyclic subgroup to connect lifts of long paths from the quotient. This gives a quantitative variant: a word of length at most \(\ell\) allows a loss of at most a factor of \(2^\ell\) in the proportion covered, provided the quotient has at least three vertices (Lemma~\ref{lem:shortcentral}).

It remains to bound the total word length used in these successive lifts. For this, we apply the Burnside basis theorem and the standard generation of lower central factors by iterated commutators (Lemma~\ref{lem:shortgenerators}). These give boundedly many words of bounded length for each distinct prime dividing \(m\), at each of the boundedly many stages. Since \(m\) has \(o(\log m)\) distinct prime divisors, the total loss is at most \(2^{o(\log m)}=m^{o(1)}\), as required.

\section{Notation and preliminaries}\label{sec:prelim}
We collect notation shared by several parts of the argument and record the elementary facts needed throughout. Our graph and group conventions are standard; see, for example, Godsil and Royle~\cite{GR} and Robinson~\cite{Robinson}. We state standard facts in the forms used below and provide references.

\subsection{Graph notation}\label{sec:graphnotation}
All graphs are finite and simple unless explicitly stated otherwise. We write \(V(X)\) and \(E(X)\) for the vertex and edge sets of \(X\), and \(u\sim v\) when \(u\) and \(v\) are adjacent. For a walk \(W\), the sets \(V(W)\) and \(E(W)\) consist of its vertices and edges. We also regard paths and cycles as their corresponding subgraphs; intersections between paths refer to their vertex sets. The \emph{order} of a path or cycle is its number of vertices, and its \emph{length} is its number of edges. A singleton vertex is a path of length zero.

An \emph{automorphism} of \(X\) is a permutation of \(V(X)\) such that two vertices are adjacent if and only if their images are adjacent. The group of all automorphisms, under composition, is denoted by \(\Aut(X)\). The graph is \emph{vertex-transitive} if, for all \(u,v\in V(X)\), some automorphism sends \(u\) to \(v\).

For a nonempty graph \(X\), let \(p(X)\) be the maximum order of a path in \(X\), and let \(c(X)\) be the maximum length of a cycle in \(X\), with \(c(X)=0\) when \(X\) has no cycle.

For \(A\subseteq V(X)\), we write \(X-A\) for the subgraph induced by \(V(X)\setminus A\). A vertex set is \emph{connected} if it induces a connected subgraph.

For a connected graph \(X\), we write \(\dist_X(u,v)\) for the distance between \(u\) and \(v\), and \(\diam(X)=\max_{u,v\in V(X)}\dist_X(u,v)\) for its diameter. The \emph{diameter of \(A\) in \(X\)}, for a nonempty set \(A\subseteq V(X)\), is \(\max_{u,v\in A}\dist_X(u,v)\). Paths used to measure this distance may leave \(A\).

\subsection{Groups and quotients}\label{sec:groupnotation}
A \emph{group} is a set with an associative operation, an identity and inverses. All groups used below are finite. We write the operation multiplicatively, denote the identity by \(1\), and call \(|G|\) the \emph{order} of \(G\). A group is \emph{trivial} if it consists only of \(1\), and \emph{nontrivial} otherwise. A \emph{subgroup} \(H\le G\) is a subset of \(G\) containing \(1\) and closed under products and inverses. The \emph{subgroup generated} by \(S\subseteq G\), written \(\langle S\rangle\), consists of all finite products of elements of \(S\) and their inverses. Throughout this paper, the \emph{rank} \(\rank(G)\) is the minimum size of a generating set of \(G\). A group generated by one element is \emph{cyclic}, and a group whose elements all commute is \emph{abelian}. The \emph{order} of an element \(g\) is the least positive integer \(m\) with \(g^m=1\).

For \(H\le G\) and \(g\in G\), the \emph{left coset} \(gH\) is \(\{gh:h\in H\}\). The \emph{index} \([G:H]\) is the number of distinct left cosets of \(H\) in \(G\). These cosets partition \(G\), so \(|G|=|H|[G:H]\). A \emph{representative} of a coset is any of its elements. A subgroup \(H\le G\) is \emph{normal in \(G\)}, written \(H\normal G\), if \(gHg^{-1}=H\) for all \(g\in G\). In this case the \emph{quotient group} \(G/H\) consists of the cosets with multiplication \((gH)(hH)=ghH\) for \(g,h\in G\). The map \(g\mapsto gH\) is the \emph{quotient homomorphism}; a homomorphism is a map preserving products. More generally, the \emph{fibre} of a map \(f:A\to B\) over \(b\in B\) is \(f^{-1}(b)=\{a\in A:f(a)=b\}\). A \emph{lift} of an element under a quotient homomorphism is an element in its fibre. A bijective homomorphism is an \emph{isomorphism}; groups related by one are \emph{isomorphic}. A quotient group is also called a \emph{factor}. For \(g\in G\), \emph{conjugation by \(g\)} is the map \(x\mapsto gxg^{-1}\) on \(G\). For a homomorphism \(\varphi:G\to K\), its \emph{kernel} is \(\ker\varphi=\{g\in G:\varphi(g)=1\}\), and its \emph{image} is \(\varphi(G)=\{\varphi(g):g\in G\}\). The \emph{centre} is \(Z(G)=\{z\in G:zg=gz\text{ for all }g\in G\}\); its elements are called \emph{central}, and a subgroup contained in \(Z(G)\) is a \emph{central subgroup}.

\subsection{Group actions and graph quotients}\label{sec:actionnotation}
An \emph{action} of a group \(\Gamma\) on a set \(V\) assigns to each \(g\in\Gamma\) a permutation \(v\mapsto gv\) of \(V\), such that \((gh)v=g(hv)\) and \(1v=v\) for all \(g,h\in\Gamma\) and \(v\in V\). For \(g\in\Gamma\) and \(A\subseteq V\), the \emph{translate} of \(A\) by \(g\) is \(gA=\{ga:a\in A\}\). An action of \(\Gamma\) on a graph \(X\) is an action on \(V(X)\) for which the permutation \(v\mapsto gv\) is an automorphism of \(X\) for every \(g\in\Gamma\). For \(v\in V\), the \emph{orbit} of \(v\) under \(H\le\Gamma\) is \(Hv=\{hv:h\in H\}\). The \emph{stabilizer} of \(v\) in \(\Gamma\) is \(\Gamma_v=\{g\in\Gamma:gv=v\}\); for an action on a graph, it is called a \emph{vertex stabilizer}. The orbit--stabilizer formula gives \(|\Gamma v|=[\Gamma:\Gamma_v]\); see~\cite[1.6.1]{Robinson}. The action is \emph{transitive} if \(V\ne\varnothing\) and, for every \(u,v\in V\), there exists \(g\in\Gamma\) with \(gu=v\). It is \emph{free} if no nonidentity element fixes any point. A free transitive action is \emph{regular}: for each ordered pair of points there is exactly one group element sending the first to the second.

If \(H\le\Aut(X)\), the \emph{orbit quotient} \(X/H\) has the \(H\)-orbits as vertices; two distinct orbits are adjacent if an edge of \(X\) joins them. The vertices mapped to a given quotient vertex form its \emph{fibre}, which here is an \(H\)-orbit. A \emph{lift} of a quotient path \(B_1,\ldots,B_k\) is a path \(v_1,\ldots,v_k\) with \(v_i\in B_i\) for each \(i\). If \(H\normal\Gamma\le\Aut(X)\), each \(g\in\Gamma\) sends the orbit \(Hv\) to \(H(gv)\). These permutations form the \emph{induced group}, a subgroup of \(\Aut(X/H)\).

We record the standard orbit-quotient facts used below; see Godsil and Royle~\cite[Section~9.3, p.~196]{GR} for the adjacency property underlying path lifting.

\begin{lemma}[Orbit quotients and path lifting]\label{lem:orbits}
Let \(X\) be a connected graph and \(H\le\Aut(X)\). The quotient \(X/H\) is connected, and every path in \(X/H\) lifts to a path of the same order in \(X\), starting at any prescribed vertex of its first orbit. If, in addition, \(H\normal\Gamma\le\Aut(X)\) and \(\Gamma\) is transitive on \(V(X)\), then \(X/H\) is vertex-transitive and the \(H\)-orbits have equal size.
\end{lemma}
\begin{proof}
Suppose two \(H\)-orbits \(B,C\) are adjacent, and choose an edge \(bc\) with \(b\in B\) and \(c\in C\). For any \(v\in B\), some \(h\in H\) satisfies \(hb=v\). Thus \(hc\in C\) is a neighbour of \(v\). We can therefore lift a quotient path successively from any prescribed starting vertex. The lifted vertices are distinct because the quotient path visits distinct orbits. Projecting paths in \(X\) and deleting consecutive repetitions gives walks between quotient vertices, proving that \(X/H\) is connected.

If \(H\normal\Gamma\), then \(g(Hv)=H(gv)\) for every \(g\in\Gamma\). Hence \(\Gamma\) permutes the orbits by graph automorphisms. Transitivity on \(V(X)\) makes this action transitive on the orbits and gives bijections between them, proving both remaining assertions.
\end{proof}

\subsection{Cayley graphs and labelled walks}\label{sec:cayleynotation}
A set \(S\subseteq G\) is \emph{symmetric} if \(s^{-1}\in S\) whenever \(s\in S\). For a symmetric generating set \(S\), the \emph{Cayley graph} \(\Cay(G,S)\) has vertex set \(G\), with \(g\) adjacent to \(gs\) for \(s\in S\) whenever \(gs\ne g\). A word in \(S\), or \(S\)-\emph{word}, is a sequence \(s_1,\ldots,s_k\) of members of \(S\); its length is \(k\), and its product is \(s_1\cdots s_k\). For \(H\normal G\), write \(S_H=\{sH:s\in S\}\) for the image of \(S\) in \(G/H\). The subscript specifies the subgroup by which we quotient. These Cayley graphs are simple: identity generators produce no edges, and repeated descriptions of an edge describe the same unordered pair.

The graph \(\Cay(G,S)\) is connected because every group element is represented by an \(S\)-word. Left multiplication by \(a\in G\) preserves edges, since it sends \(g,gs\) to \(ag,ags\). The action of \(G\) on itself by left multiplication is regular: for any \(g,h\in G\), there is exactly one \(a\in G\) with \(ag=h\), namely \(a=hg^{-1}\). If \(H\normal G\), then \(Hg=gH\), and an edge \(g,gs\) projects to \(gH,gsH\). Conversely every edge of \(\Cay(G/H,S_H)\) arises in this way. The orbit quotient by \(H\) is therefore \(\Cay(G/H,S_H)\).

Let \(S\) be a symmetric generating set of \(G\), and let \(H\normal G\). A \emph{labelled walk} in \(G/H\) consists of a starting coset \(gH\) and a word \(s_1,\ldots,s_k\) in the original set \(S\). Its successive vertices are \(gH,gs_1H,\ldots,gs_1\cdots s_kH\). Recording the original labels matters because different generators can have the same image in \(G/H\). Steps that repeat a quotient vertex are allowed in this labelled object. The walk is \emph{closed} if its last vertex equals its first. For \(k\ge1\), it is a \emph{simple labelled cycle of length \(k\)} if its last vertex equals its first and its first \(k\) vertices are distinct. We allow \(k=1\) and \(k=2\) for labelled cycles, although ordinary cycles in our simple graphs have length at least three. Reversal replaces the labels by \(s_k^{-1},\ldots,s_1^{-1}\), which still lie in \(S\) because it is symmetric. A labelled cycle of length two may use two different labels over the same underlying edge.

The \emph{lift} of a labelled walk starting at \(g\in G\) is obtained by multiplying successively by the same labels. Following the terminology of Ghaderpour and Witte Morris~\cite{GM}, the \emph{voltage} of a closed labelled walk in \(G/H\) is the ordered product \(h\in H\) of its labels in \(G\). Thus a lift starting at \(g\) ends at \(gh\). When \(H\le Z(G)\), we also have \(gh=hg\). Left multiplication by \(H\) preserves the labels and acts regularly on each fibre. A cyclic change of starting point conjugates the label product; when \(H\le Z(G)\), centrality therefore makes the voltage unchanged. Changing the representative of the starting coset also leaves the label product unchanged. We use voltage only with the labels and their order specified.

\subsection{Nilpotent groups}\label{sec:nilpotentnotation}
For \(u,v\in G\), the \emph{commutator} is \([u,v]=u^{-1}v^{-1}uv\). For an integer \(j\ge1\) and \(u_1,\ldots,u_j\in G\), define the \emph{simple commutator of weight \(j\)} recursively by \([u_1]=u_1\) and, for \(j\ge2\), \([u_1,\ldots,u_j]=[[u_1,\ldots,u_{j-1}],u_j]\). Thus the weight counts the entries, with brackets taken from the left.

For subgroups \(A,B\le G\), let \([A,B]\) be generated by the elements \([a,b]\) with \(a\in A\) and \(b\in B\). The \emph{lower central series} is \(\gamma_1(G)=G\) and \(\gamma_{j+1}(G)=[\gamma_j(G),G]\). Its subgroups are its \emph{terms}, and the quotients \(\gamma_j(G)/\gamma_{j+1}(G)\) are its \emph{lower central factors}. The series is descending, each term is normal in \(G\), and
\[
 \gamma_j(G)/\gamma_{j+1}(G)\le Z(G/\gamma_{j+1}(G))\qquad(j\ge1);
\]
see~\cite[Section~5.1, p.~125]{Robinson}. For a nonnegative integer \(c\), the group \(G\) is \emph{nilpotent of class at most \(c\)} if \(\gamma_{c+1}(G)=\{1\}\); its class is the least such integer. The trivial group has class zero.

A quotient of \(G\) has rank at most \(\rank(G)\), since the images of any generating set generate the quotient. For every normal subgroup \(H\normal G\),
\[
 \gamma_j(G/H)=\{gH:g\in\gamma_j(G)\}\qquad(j\ge1).
\]
Thus taking a quotient does not increase nilpotency class; see~\cite[Theorem~5.1.4]{Robinson}. In particular, if \(G\) has class at most \(c\ge1\), then \(G/\gamma_c(G)\) has class at most \(c-1\).

\subsection{Numerical and asymptotic notation}\label{sec:asymptoticnotation}
We write \(\Z\) for the integers, \([k]=\{1,\ldots,k\}\), and \(|A|\) for the cardinality of a finite set \(A\). We use \(\log\) for the natural logarithm. For a real number \(x\), \(\lfloor x\rfloor\) and \(\lceil x\rceil\) denote its floor and ceiling. For integers \(m\ge1\) and \(h\ge0\), \(\gcd(m,h)\) denotes their greatest common divisor, with \(\gcd(m,0)=m\). The group of integers modulo \(m\), under addition, is denoted by \(\Z/m\Z\). Empty sums and products have values zero and one, respectively; an empty word in a group has product \(1\).

For functions \(G(n)>0\), the notation \(F(n)=O(G(n))\) means that \(|F(n)|\le C G(n)\) for a constant \(C>0\) and all sufficiently large \(n\); \(F(n)=\Omega(G(n))\) means \(F(n)\ge cG(n)\) for a constant \(c>0\) and all sufficiently large \(n\). The notation \(h(n)=o(G(n))\), for \(G(n)>0\), means \(h(n)/G(n)\to0\) as \(n\to\infty\); in particular, \(h(n)=o(1)\) means \(h(n)\to0\). A lower bound \(F(n)\ge n^{\alpha-o(1)}\) means that for every \(\eps>0\) there is \(n_0(\eps)\) such that \(F(n)\ge n^{\alpha-\eps}\) whenever \(n\ge n_0(\eps)\). An upper bound \(F(n)\le n^{\alpha+o(1)}\) means that for every \(\eps>0\), one has \(F(n)\le n^{\alpha+\eps}\) for all sufficiently large \(n\). A bound is \emph{uniform} over a family if its constants and thresholds depend only on the parameters explicitly fixed for that family. We call a lower bound of the form \(n^{1-o(1)}\) \emph{near-linear}.

\section{Paths through large orbits}\label{sec:routes}
We construct a long path when the orbits of a normal subgroup are large relative to their diameters measured in the original graph. We repeatedly traverse a spanning tree of maximum degree at most three in the quotient and join the incoming and outgoing edges at each visit by a random short path in the original graph. The degree bound limits the number of visits to each orbit in one traversal to at most four, independently of its degree in the quotient. Vertex-transitivity and the Lov\'asz local lemma allow us to choose these paths so that those from nonconsecutive visits are disjoint. We then extract a long simple path from their union together with the edges between successive visits. Lemma~\ref{lem:largefibres} gives the resulting bound.

\begin{definition}[Random paths between vertices in an orbit]\label{def:routes}
Let \(X\) be a connected graph with a transitive group \(\Gamma\le\Aut(X)\), and let \(H\normal\Gamma\). Denote the family of \(H\)-orbits by \(\mathcal B\), their number by \(q\), and their common size by \(s\), so \(n:=|V(X)|=qs\). Suppose each orbit has diameter at most a real number \(r\ge0\), measured in \(X\), and set \(R=r+1\). Choose two vertices independently and uniformly at random from an orbit \(B\), then choose a path uniformly at random from all shortest paths in \(X\) between them. If the endpoints coincide, use the singleton path. For brevity, we call this random path a \emph{route}. For \(v\in V(X)\), write \(\mu_B(v)\) for the probability that it contains \(v\). Its endpoints lie in \(B\), but the path may leave \(B\).
\end{definition}

\begin{lemma}[Lifting through large orbits]\label{lem:largefibres}
In the setting of Definition~\ref{def:routes}, if \(q\ge2\) and \(s\ge192R^2\), then
\begin{equation}\label{eq:largefibres}
 p(X)\ge\frac{n}{1000R^2}.
\end{equation}
\end{lemma}

\subsection{A spanning tree of bounded degree}
We first show that the orbit quotient \(X/H\) from Definition~\ref{def:routes} has a spanning tree of maximum degree at most three. A graph \(X\) is \emph{1-tough} if, for every \(A\subseteq V(X)\) such that \(X-A\) has at least two components, their number is at most \(|A|\). We use the following consequence of Win's theorem.

\begin{theorem}[Win~\cite{Win}]\label{thm:Win}
Every finite connected 1-tough graph has a spanning tree of maximum degree at most three.
\end{theorem}

This is the maximum-degree-three case of Win's toughness criterion. For a connected graph of order one or two, the graph itself is the required spanning tree.

The \emph{edge-connectivity} of a connected graph of order at least two is the minimum number of edges whose deletion disconnects it. Mader's edge-connectivity theorem~\cite{Mader} states that this number equals the degree for a connected vertex-transitive graph. The next lemma follows from this theorem and Win's theorem; we give the short deduction of 1-toughness.

\begin{lemma}\label{lem:tree}
Every connected vertex-transitive graph on at least two vertices is 1-tough and has a spanning tree of maximum degree at most three.
\end{lemma}
\begin{proof}
Let \(X\) be such a graph, and let \(d\ge1\) be its common degree. By Mader's edge-connectivity theorem~\cite{Mader}, every nonempty proper vertex set has at least \(d\) edges to its complement. Suppose that \(X-U\) has \(h\ge2\) components. Each component has at least \(d\) edges to \(U\), whereas at most \(d|U|\) edges join \(U\) to its complement. Thus \(hd\le d|U|\), and hence \(h\le|U|\). This proves 1-toughness. The spanning tree follows from Theorem~\ref{thm:Win}.
\end{proof}

\subsection{Controlling intersections of the random paths}\label{sec:probabilitynotation}
We write \(\Pr(A)\) for the probability of an event \(A\), and \(\E Z\) for the expectation of a random variable \(Z\). Events \(A,B\) are \emph{independent} if \(\Pr(A\cap B)=\Pr(A)\Pr(B)\). A \emph{dependency graph} for a family of events \((A_i)_{i\in I}\) is a graph on \(I\) such that, for each \(i\in I\), the event \(A_i\) is independent of every event formed by unions, intersections and complements of events indexed outside \(\{i\}\cup N(i)\), where \(N(i)\) denotes the set of neighbours of \(i\). We use the following form of the Lov\'asz local lemma.

\begin{lemma}[Lov\'asz local lemma, probability-sum form~{\cite[Corollary~6.1.10]{ZhaoPM}}]\label{cor:LLLsum}
Let \((A_i)_{i\in I}\) be a finite family of events with a dependency graph. If \(\Pr(A_i)<1/2\) and
\[
 \sum_{j\in N(i)}\Pr(A_j)\le\frac14
 \qquad(i\in I),
\]
then with positive probability none of the events occurs.
\end{lemma}

The next estimate is where vertex-transitivity enters the probability calculation. If we take one random route with endpoints in each orbit, the expected number of these routes containing a given vertex is the same at every vertex, even though the routes may leave their endpoint orbits.

\begin{lemma}\label{lem:traffic}
In the setting of Definition~\ref{def:routes}, for every \(v\in V(X)\),
\begin{equation}\label{eq:traffic}
 \sum_{B\in\mathcal B}\mu_B(v)\le R/s.
\end{equation}
\end{lemma}
\begin{proof}
An automorphism \(g\in\Gamma\) maps the random route distribution for \(B\) to that for \(gB\): it bijects the possible endpoint pairs and the shortest paths for each pair. Since \(\Gamma\) permutes \(\mathcal B\), the left side of~\eqref{eq:traffic} has the same value at \(v\) and \(gv\). Transitivity makes this value constant over all \(n\) vertices. For each \(B\in\mathcal B\), let \(P_B\) denote a random route with endpoints in \(B\). Each route contains at most \(R\) vertices, so
\[
 \sum_{v\in V(X)}\sum_{B\in\mathcal B}\mu_B(v)
 =\sum_{B\in\mathcal B}\E|V(P_B)|\le qR.
\]
Division by \(n=qs\) proves the claim.
\end{proof}

\subsection{Proof of the large-orbit bound}
An \emph{oriented edge} is an ordered pair of adjacent vertices; its first vertex is its \emph{tail} and its second its \emph{head}.

\begin{proof}[Proof of Lemma~\ref{lem:largefibres}]
\emph{Choosing the visits.}
The quotient \(Q=X/H\) is connected and vertex-transitive by Lemma~\ref{lem:orbits}. By Lemma~\ref{lem:tree}, it has a spanning tree of maximum degree at most three. Set
\[
 t=\left\lfloor\frac{s}{192R^2}\right\rfloor.
\]
Then \(t\ge1\) and \(t\ge s/(384R^2)\).
Traverse each tree edge once in each direction, and repeat this closed traversal \(t\) times. The resulting walk is
\[
 B_1,\ldots,B_m,\qquad m=2t(q-1)+1.
\]
In one traversal, each orbit occurs as the initial vertex of an edge exactly as many times as its tree degree. Thus each orbit occurs at most \(3t+1\) times in the displayed walk.

\emph{Joining successive visits.}
For each \(i<m\), independently choose an edge \(E_i\) uniformly at random from the edges between \(B_i\) and \(B_{i+1}\), and orient it from \(B_i\) to \(B_{i+1}\). Its tail is uniformly distributed on \(B_i\), and its head is uniformly distributed on \(B_{i+1}\). Indeed, \(H\) preserves the two orbits and acts transitively on each, so every vertex on either side has the same number of neighbours on the other side. Also choose the initial and terminal vertices uniformly at random from \(B_1\) and \(B_m\), respectively, independently of one another and of the edges.

At visit \(i\), join the incoming and outgoing endpoints by a shortest path \(P_i\) chosen uniformly at random, using the initial or terminal vertex at the two ends of the walk. Conditional on the endpoints, make all these path choices independently. The two endpoints of \(P_i\) are independent and uniformly distributed in \(B_i\), so \(P_i\) has the distribution from Definition~\ref{def:routes}.

For the dependence calculation, independently of all edge and endpoint choices, preselect for every visit \(i\) and every possible ordered pair of endpoints a shortest path between these endpoints uniformly at random, making all these selections independently. Then \(P_i\) is determined by the paths preselected for visit \(i\) and the edges \(E_{i-1},E_i\), with the extra endpoint variable replacing the missing edge variable at either end. In particular, \(P_i\) and \(P_j\) are independent whenever \(|i-j|\ge2\).

\emph{Avoiding intersections.}
For each unordered pair with \(|i-j|\ge2\), let \(A_{ij}=A_{ji}\) be the event that \(P_i\) and \(P_j\) intersect, and put \(p_{ij}=\Pr(A_{ij})\). For each fixed visit, we bound the sum of its intersection probabilities with nonconsecutive visits. The union bound states that the probability of a union of events is at most the sum of their probabilities. Independence, this bound, and Lemma~\ref{lem:traffic} give
\begin{align}
 \sum_{j:\,|i-j|\ge2}p_{ij}
 &\le\sum_{v\in V(X)}\mu_{B_i}(v)
              \sum_{j:\,|i-j|\ge2}\mu_{B_j}(v)\notag\\
 &\le(3t+1)\frac Rs
              \sum_{v\in V(X)}\mu_{B_i}(v)
 \le\frac{(3t+1)R^2}{s}=:a.\label{eq:rowsum}
\end{align}

Here \(\mu_{B_j}(v)\) depends on the visit \(j\) only through its orbit \(B_j\): whenever \(B_j=B_k\), the routes \(P_j\) and \(P_k\) have the same distribution. Since each orbit occurs at most \(3t+1\) times, its term in the sum from Lemma~\ref{lem:traffic} is repeated at most \(3t+1\) times. Also, \(\sum_v\mu_{B_i}(v)=\E|V(P_i)|\le R\). Our choice of \(t\) ensures
\[
 a\le\frac3{192}+\frac1{192}=\frac1{48}.
\]

Join two events in a dependency graph whenever at least one of the independent random choices described above is used to determine both. A neighbour \(A_{k\ell}\) of \(A_{ij}\) must have at least one index in
\[
 \{i-1,i,i+1,j-1,j,j+1\}\cap[m].
\]
Events outside this neighbourhood collectively depend on variables disjoint from those determining \(A_{ij}\), as required for a dependency graph. Summing~\eqref{eq:rowsum} over these at most six indices gives
\[
 \sum_{A_{k\ell}\sim A_{ij}}\Pr(A_{k\ell})
 \le6a\le\frac18.
\]
Every bad event has probability at most \(a<1/2\), and the sum over its neighbours is less than \(1/4\). Lemma~\ref{cor:LLLsum} therefore gives an outcome with
\[
 P_i\cap P_j=\varnothing\qquad\text{whenever }|i-j|\ge2.
\]

\emph{Extracting a long path.}
Consecutive routes may still intersect, so we extract a simple path from their union. Let \(Y\) consist of the chosen paths and connecting edges \(E_i\), with no other edges of \(X\). This is a connected subgraph. Label each vertex of \(Y\) by the smallest index of a path containing it. A vertex of \(P_i\) has label \(i-1\) or \(i\). Thus an edge of a chosen path changes the label by at most one, and a connecting edge changes it by at most two.

The initial vertex has label \(1\), while the terminal vertex has label at least \(m-1\). They are distinct, since \(m\ge3\) and \(P_1\cap P_m=\varnothing\). Any simple path between them in \(Y\) has at least \((m-2)/2\) edges, and hence at least \(t(q-1)\) vertices. Consequently,
\[
 p(X)\ge t(q-1)
 \ge\frac{s}{384R^2}\cdot\frac q2
 =\frac{n}{768R^2}
 \ge\frac{n}{1000R^2},
\]
as required.
\end{proof}

\section{Lifting cycles from quotient Cayley graphs}\label{sec:cyclic}
We build long cycles in a Cayley graph by lifting from a quotient by a cyclic subgroup of the centre of the underlying group. If a generator of this subgroup can be written as a word of length \(\ell\) in the given generators, Lemma~\ref{lem:shortcentral} shows that lifting retains at least \(1/2^\ell\) of the proportion of vertices covered in the quotient, provided the quotient has at least three vertices. In Section~\ref{sec:nilpotent}, we apply Lemma~\ref{lem:shortcentral} along a sequence of quotients by cyclic central subgroups to obtain the near-linear bound for nilpotent Cayley graphs.

Recall that \(c(X)\) is the maximum length of an ordinary cycle in \(X\), and equals zero if \(X\) has no cycle. For a nonempty graph \(X\), write
\begin{equation}\label{eq:cycleproportion}
 \rho(X)=\frac{c(X)}{|V(X)|}.
\end{equation}
Recall that the \emph{voltage} of a closed labelled walk in \(G/H\) is the ordered product \(h\in H\) of its original labels in \(S\). A lift starting at \(g\in G\) ends at \(gh\); thus the voltage records its displacement within the starting fibre.

Our starting point is the classical Factor Group Lemma, recorded in Witte and Gallian~\cite[Section~2.2]{WitteGallian}; see also Ghaderpour and Witte Morris~\cite{GM}. If a Hamiltonian cycle in a quotient Cayley graph has voltage generating the cyclic kernel, repeating its label sequence as many times as the order of the kernel gives a Hamiltonian cycle in the original graph.

The same observation applies to a cycle that need not be Hamiltonian. Suppose that \(H=\langle z\rangle\le Z(G)\) has order \(m\), and that a simple labelled cycle of length \(a\) in \(\Cay(G/H,S_H)\) has voltage \(z\). Its \(m\) successive traversals visit every vertex above that cycle exactly once before returning: each traversal starts at the next translate by \(z\), and the distinct quotient vertices distinguish positions within a traversal. Thus the lift is an ordinary cycle of length \(ma\) whenever \(ma\ge3\). Here \(S_H=\{sH:s\in S\}\), and we retain the original generator labels, including when quotient steps repeat a vertex or describe the same edge.

We develop a quantitative variant in which the length of the resulting cycle is compared with the length of a longest cycle in the quotient, whose voltage need not generate the kernel. In the main case of Lemma~\ref{lem:cyclicextension}, a short simple labelled cycle supplies the generating voltage. Lemma~\ref{lem:longear} supplies a long path with distinct endpoints on this short cycle and all internal vertices outside it. Lemma~\ref{lem:thetalifting} then uses the two arcs of the short cycle to join at least half of the lifts of this path into one cycle. Finally, Lemma~\ref{lem:shortcentral} extends the bound to a central element represented by a short word, with the loss in the proportion of vertices covered bounded in terms of the word length.

We will use the following elementary connectivity fact.

\begin{lemma}\label{lem:twoconnected}
Every connected vertex-transitive graph on at least three vertices is 2-connected and contains a cycle. For every nonempty graph \(X\), we have \(p(X)\ge c(X)\).
\end{lemma}
\begin{proof}
Deleting a leaf of a spanning tree leaves the graph connected. By vertex-transitivity, deleting any vertex does so, proving 2-connectivity. Such a graph contains a cycle. Deleting an edge of any cycle gives a path on the same vertices.
\end{proof}

\subsection{Finding and lifting a long ear}

The \emph{internal vertices} of a path are its vertices other than
its endpoints. An \emph{arc} between two distinct vertices of a
cycle is either of the two paths along the cycle joining them.
For an ordinary or simple labelled cycle \(C\), a \(C\)-\emph{ear}
is a path with distinct endpoints on \(C\) and all internal
vertices outside \(C\). For a simple labelled cycle of length two,
its two labelled edges serve as the two arcs. Paths are
\emph{vertex-disjoint} if their vertex sets are disjoint.

We first use a cycle \(D\) longer than \(C\) to find a long
\(C\)-ear \(P\). When the voltage of \(C\) generates the central
cyclic subgroup, we then use its two arcs to join at least half
of the lifts of \(P\) into one cycle. The following form of
Menger's theorem is used in the first step.

\begin{theorem}[Menger~\cite{Menger}]\label{thm:Menger}
If \(A,B\) are disjoint vertex sets, each of size at least two,
in a 2-connected graph, there are two vertex-disjoint
\(A\)--\(B\) paths, with distinct endpoints in each set.
\end{theorem}

\begin{lemma}[Finding a long ear]\label{lem:longear}
Let \(Q\) be a 2-connected graph, let \(C\) be an ordinary or
simple labelled cycle of length \(a\ge2\) in \(Q\), and let
\(D\) be an ordinary cycle in \(Q\) of length \(L>a\).
Then there is a \(C\)-ear \(P\) with
\begin{equation}\label{eq:lonear}
 |E(P)|\ge L/a.
\end{equation}
\end{lemma}
\begin{proof}
Put \(b=|V(C)\cap V(D)|\). If \(b\ge2\), the \(b\) segments
of \(D\) between successive vertices of \(C\) have total length
\(L\). A longest segment is a \(C\)-ear of length at least
\(L/b\ge L/a\).

If \(b=0\), Theorem~\ref{thm:Menger} gives two vertex-disjoint
\(V(C)\)--\(V(D)\) paths. Truncate each at its first vertex of
\(D\) and its last preceding vertex of \(C\), so their interiors
avoid both cycles. Joining them by the longer arc of \(D\)
gives a \(C\)-ear of length at least \(L/2\).

If \(b=1\), write \(V(C)\cap V(D)=\{v\}\). Since \(Q-v\) is
connected, it contains a shortest path between
\(V(C)\setminus\{v\}\) and \(V(D)\setminus\{v\}\). Its interior
avoids both cycles. Append the longer arc of \(D\) from its
endpoint to \(v\), obtaining a \(C\)-ear of length at least
\(L/2\). Since \(a\ge2\), both remaining cases
imply~\eqref{eq:lonear}.
\end{proof}

We now turn to lifting the ear. Suppose that the voltage of
\(C\) generates the central cyclic subgroup. Lifts of the two
arcs of \(C\) connect the endpoints of the lifted ears.
We will select these connections to obtain one cycle containing
at least half of the lifted ears.

Related cycle-joining arguments appear in Ghaderpour and
Witte Morris~\cite[Section~5]{GM}, who join cycles while
controlling their voltages to construct quotient Hamiltonian
cycles to which the Factor Group Lemma applies. Here we need
to control the number of lifted ears retained. The following
elementary lemma gives this count: in its application, the
\(A_i\)- and \(B_i\)-edges represent lifts of the two arcs of
\(C\), and the \(P_i\)-edges represent lifts of \(P\).

\begin{lemma}\label{lem:cyclicmatching}
Let \(m\ge2\) and \(0\le h<m\) be integers. Form an undirected
labelled multigraph with vertices
\[
 \{x_i,y_i:i\in\Z/m\Z\}
\]
and edges
\[
 A_i=x_iy_i,\qquad
 B_i=y_ix_{i+1},\qquad
 P_i=x_iy_{i+h}.
\]
Edges with different labels are retained even when they join
the same pair of vertices. Then this graph contains a simple
cycle of length at least four using at least \(m/2\) of the
\(P_i\)'s.
\end{lemma}
\begin{proof}
If \(h=0\), the \(P_i\)'s and \(B_i\)'s form a cycle of length
\(2m\) containing every \(P_i\), so assume \(h\ne0\).

Put \(d=\gcd(m,h)\). Following \(P_i\) and then \(A_{i+h}\)
takes \(x_i\) to \(x_{i+h}\). Thus the \(A_i\)'s and \(P_i\)'s
form \(d\) disjoint cycles
\(\mathcal C_0,\ldots,\mathcal C_{d-1}\), where
\(\mathcal C_j\) contains precisely the vertices \(x_i,y_i\)
with \(i\equiv j\pmod d\). Each cycle contains \(m/d\) of the
\(P_i\)'s. If \(d=1\), this already gives the required cycle.

Suppose \(d\ge2\). For \(j=0,\ldots,d-2\), delete \(A_j\)
from \(\mathcal C_j\), leaving an \(x_j\)--\(y_j\) path \(R_j\)
that contains every \(P_i\) in that component. The edges
\(B_0,\ldots,B_{d-3}\) join these paths in order into an
\(x_0\)--\(y_{d-2}\) path; when \(d=2\), this is just \(R_0\).

Both \(x_{d-1}\) and \(y_{m-1}\) lie on the remaining cycle
\(\mathcal C_{d-1}\). The edges
\(B_{d-2}=y_{d-2}x_{d-1}\) and \(B_{m-1}=y_{m-1}x_0\)
connect the endpoints of the preceding path to these vertices.
Closing it along either arc of \(\mathcal C_{d-1}\) gives
a simple cycle, since the path is disjoint from
\(\mathcal C_{d-1}\).

This cycle contains every \(P_i\) in the first \(d-1\)
components, hence at least
\[
 \frac{d-1}{d}m\ge\frac m2
\]
of the \(P_i\)'s. Since \(h\ne0\), we have \(m/d\ge2\), so
it contains at least two \(P_i\)-edges with disjoint endpoints
and therefore has length at least four.
\end{proof}

We now identify the graph in Lemma~\ref{lem:cyclicmatching}
among the lifted paths.

\begin{lemma}[Lifting an ear through a cyclic quotient]
\label{lem:thetalifting}
Let \(X=\Cay(G,S)\), and let \(H\le Z(G)\) be cyclic of order
\(m\ge2\). Suppose that \(C\) is a simple labelled cycle of
length at least two in \(Q=\Cay(G/H,S_H)\), with voltage
generating \(H\). If \(P\) is a \(C\)-ear with at least two
edges, then \(X\) contains a simple cycle containing at least
\(m/2\) vertex-disjoint lifts of \(P\). In particular,
\begin{equation}\label{eq:thetalifting}
 c(X)\ge\frac m2\bigl(|E(P)|+1\bigr).
\end{equation}
\end{lemma}
\begin{proof}
Let \(x,y\) be the endpoints of \(P\). Orient \(C\), starting
at \(x\), write its voltage as \(z\), and denote its two
oriented arcs by
\[
 x\xrightarrow{A}y,\qquad y\xrightarrow{B}x.
\]
Choose a lift \(x_0\) of \(x\), and let \(y_0\) be the endpoint
of the lift of \(A\) starting at \(x_0\). Since \(z\) generates
\(H\), the fibres above \(x\) and \(y\) are
\[
 x_i=z^ix_0,\qquad y_i=z^iy_0
 \qquad(i\in\Z/m\Z).
\]

Choose labels for \(P\), oriented from \(x\) to \(y\). Its lift
from \(x_0\) ends at \(y_h\) for a unique integer \(0\le h<m\).
Translation by powers of \(z\) gives the lifted paths
\[
 x_i\xrightarrow{A_i}y_i,\qquad
 y_i\xrightarrow{B_i}x_{i+1},\qquad
 x_i\xrightarrow{P_i}y_{i+h}.
\]
Distinct lifts of a fixed labelled path are vertex-disjoint:
a common vertex would occur at the same position in the
quotient path, and reversing the preceding labels would give
the same starting vertex. Also, the interiors of \(A,B,P\)
in the quotient are pairwise disjoint and avoid \(x,y\).
Hence the interiors of the displayed lifted paths are
pairwise disjoint and avoid the fibres above \(x\) and \(y\).

Treating these lifted paths as single labelled edges gives
exactly the graph in Lemma~\ref{lem:cyclicmatching}. That lemma
provides a cycle using at least \(m/2\) of the \(P\)-edges.
Replacing its edges by the corresponding lifted paths gives
a simple cycle in \(X\) containing the selected lifts of \(P\).
These lifts are vertex-disjoint and each has \(|E(P)|+1\)
vertices, proving~\eqref{eq:thetalifting}.
\end{proof}

\subsection{Lifting over a central cyclic subgroup}
Lemmas~\ref{lem:longear} and~\ref{lem:thetalifting} give the main case of the following bound.

\begin{lemma}[Lifting over a cyclic subgroup of the centre]\label{lem:cyclicextension}
Let \(X=\Cay(G,S)\), let \(H\le Z(G)\) be cyclic of order \(m\), and put \(Q=\Cay(G/H,S_H)\). Suppose that \(|V(Q)|\ge3\). If a simple labelled cycle of length \(a\) in \(Q\) has voltage generating \(H\), then
\begin{equation}\label{eq:cyclicextension}
 c(X)\ge\frac{m}{2a}\,c(Q).
\end{equation}
\end{lemma}
\begin{proof}
The case \(m=1\) is immediate, since then \(X=Q\). Assume \(m\ge2\), and let \(C\) be the specified labelled cycle.

Suppose first that \(a\ge2\). If \(c(Q)\le a\), repeating \(C\) gives a cycle of length \(ma\ge4\), which proves the bound. Otherwise, \(Q\) is 2-connected by Lemma~\ref{lem:twoconnected}. Apply Lemma~\ref{lem:longear} to \(C\) and a longest cycle of \(Q\) to obtain a \(C\)-ear \(P\) with \(|E(P)|\ge c(Q)/a>1\). Lemma~\ref{lem:thetalifting} gives
\[
 c(X)\ge\frac m2\bigl(|E(P)|+1\bigr)
 \ge\frac{m}{2a}c(Q).
\]

It remains to treat \(a=1\). Its label is a central generator \(z\in S\) of \(H\). Lift a longest path in \(Q\) to vertices \(g_1,\ldots,g_k\), where \(k=p(Q)\ge3\), and put \(m'=2\lfloor m/2\rfloor\). The vertices
\[
 z^ug_i\qquad(0\le u<m',\ 1\le i\le k)
\]
form an \(m'\)-by-\(k\) rectangular grid as a subgraph of \(X\): horizontal edges are translates of the lifted path, and vertical edges come from multiplication by \(z\). Since \(m'\) is even, this grid has a Hamiltonian cycle. Indeed, start at the top of the first column, traverse the remaining columns back and forth through successive rows, and return up the first column. Thus
\[
 c(X)\ge m'k\ge\frac m2p(Q)\ge\frac m2c(Q),
\]
as required.
\end{proof}

\subsection{Lifting from a short word for a central element}
Lemma~\ref{lem:cyclicextension} applies when a generator of the cyclic kernel occurs as the voltage of a simple labelled cycle in the quotient. We now extend this to a central element represented by a short \(S\)-word, even if that element does not itself belong to \(S\). Projecting the word to the quotient gives a closed labelled walk, possibly with repeated vertices. We decompose this walk into simple labelled cycles and quotient successively by the subgroups generated by their voltages, applying Lemma~\ref{lem:cyclicextension} at each step. If the original word has length at most \(\ell\), the total length of these cycles is at most \(\ell\), and the losses multiply to at most \(2^\ell\).

Recall that \(\rho(X)=c(X)/|V(X)|\) is the proportion of vertices on a longest cycle in \(X\).
\begin{lemma}[Lifting from a short word for a central element]\label{lem:shortcentral}
Let \(X=\Cay(G,S)\), and suppose that \(z\in Z(G)\) is a product of at most \(\ell\) elements of \(S\), for an integer \(\ell\ge0\). Put \(H=\langle z\rangle\) and \(Q=\Cay(G/H,S_H)\), and suppose that \(|V(Q)|\ge3\). Then
\begin{equation}\label{eq:shortcentral}
 \rho(X)\ge2^{-\ell}\rho(Q).
\end{equation}
\end{lemma}
\begin{proof}
If \(z=1\), then \(X=Q\) has at least three vertices, so~\eqref{eq:shortcentral} is immediate. Otherwise, project a word of length at most \(\ell\) representing \(z\) to \(G/H\). From this closed labelled walk, repeatedly delete a shortest nonempty segment whose initial and final quotient vertices coincide, until no labels remain. Each deleted segment is a simple labelled cycle. Let its word be \(w_i\), its length \(\ell_i\), and its voltage \(h_i\), for \(1\le i\le t\). Then
\[
 h_i\in H\le Z(G),\qquad
 z=h_1\cdots h_t,\qquad
 \sum_{i=1}^t\ell_i\le\ell.
\]
The product identity holds because the product of a deleted segment is central: deleting it removes a factor \(h_i\) from the product of the current word. In particular, the elements \(h_i\) generate \(H\).

Set \(H_0=\{1\}\), \(H_i=\langle h_1,\ldots,h_i\rangle\), and
\[
 X_i=\Cay(G/H_i,S_{H_i})\qquad(0\le i\le t).
\]
Thus \(X_0=X\) and \(X_t=Q\), and every \(X_i\) has at least three vertices. Omit steps with \(H_i=H_{i-1}\). At every remaining step, the word \(w_i\) gives a simple labelled cycle of length \(\ell_i\) in \(G/H_i\): its vertices before the final return are distinct modulo \(H\), hence also modulo \(H_i\le H\). In the map \(G/H_{i-1}\to G/H_i\), its voltage generates the central cyclic kernel \(H_i/H_{i-1}\). Dividing~\eqref{eq:cyclicextension} by the group orders gives
\begin{equation}\label{eq:wordstep}
 \rho(X_{i-1})\ge\frac{1}{2\ell_i}\rho(X_i).
\end{equation}

Multiplying~\eqref{eq:wordstep} and using \(2u\le2^u\) for positive integers \(u\), the total loss is at most
\[
 \prod_i2\ell_i\le2^{\sum_i\ell_i}\le2^\ell,
\]
where the product is over the nontrivial steps. This proves~\eqref{eq:shortcentral}.
\end{proof}

\section{Long cycles in nilpotent Cayley graphs}\label{sec:nilpotent}
Every connected Cayley graph of an abelian group of order at least three is Hamiltonian~\cite{Marusic}. We use this as the base case of an induction on nilpotency class. At each step, we quotient by the last nontrivial lower central term, find a long cycle by induction, and lift it through successive cyclic central quotients using Lemma~\ref{lem:shortcentral}. We first show that the required cyclic kernels have short generators.

For a positive integer \(m\), write \(\nu(m)\) for its number of distinct prime divisors, with \(\nu(1)=0\). The next lemma combines the Burnside basis theorem with the standard generation theorem for lower central factors. It gives a generating set for each lower central factor with both the number of generators and their word lengths controlled.

\begin{lemma}[Short generators for lower central factors]\label{lem:shortgenerators}
Let \(G\) be a finite nilpotent group of rank at most \(r\), where \(r\ge1\), and let \(S\) be a symmetric generating set of \(G\). For every integer \(j\ge1\), the quotient
\[
 A_j=\gamma_j(G)/\gamma_{j+1}(G)
\]
is generated by the images of at most \(r^j\nu(|G|)\) weight-\(j\) simple commutators in \(S\). Each selected commutator is represented by an \(S\)-word of length at most
\[
 \lambda_j=3\cdot2^{j-1}-2.
\]
\end{lemma}
\begin{proof}
The case \(G=\{1\}\) is immediate. A finite nilpotent group is
the direct product of its Sylow subgroups~\cite[Theorem~5.2.4]{Robinson}.
Thus the coordinates multiply independently when we identify
\[
 G=\prod_{p\mid|G|}P_p,
\]
where \(P_p\) is the unique Sylow \(p\)-subgroup of \(G\), that is, the subgroup whose order is the largest power of \(p\) dividing \(|G|\). Projection onto this factor shows that its rank is at most \(r\). The Burnside basis theorem~\cite[Theorem~5.3.2]{Robinson} implies that every generating set of a finite \(p\)-group, a group of prime-power order, contains a generating subset of minimum size. Applied to the projection of \(S\), it therefore gives a set \(S_p\subseteq S\) of at most \(r\) elements whose projections generate \(P_p\).

Fix \(j\ge1\). The at most \(r^j\) weight-\(j\) simple commutators in \(S_p\) project to generators of \(\gamma_j(P_p)/\gamma_{j+1}(P_p)\), by~\cite[Theorem~5.2.5 and the following paragraph]{Robinson}. Commutators are computed componentwise, so
\[
 A_j\cong\prod_{p\mid|G|}\bigl(\gamma_j(P_p)/\gamma_{j+1}(P_p)\bigr).
\]
Let \(B\le A_j\) be generated by the images of all these commutators, for all \(p\mid|G|\). Then \(B\) projects onto every displayed factor. Their orders are pairwise coprime and each divides \(|B|\), so \(|A_j|\) divides \(|B|\). Hence \(B=A_j\), and at most \(r^j\nu(|G|)\) commutators have been used.

Finally, expanding a simple commutator gives word-length bounds \(L_1=1\) and \(L_{k+1}=2L_k+2\) for \(k\ge1\). Thus \(L_j=3\cdot2^{j-1}-2=\lambda_j\). All letters and their inverses lie in \(S\), since \(S\) is symmetric.
\end{proof}

We now apply the lemma to bound the total loss in these successive lifts. Recall that \(\rho(X)=c(X)/|V(X)|\) is the proportion of vertices on a longest cycle. For positive integers \(r,c\), put
\[
 K_{r,c}=\sum_{j=2}^c r^j\lambda_j,
 \qquad \lambda_j=3\cdot2^{j-1}-2,
\]
so that \(K_{r,1}=0\). The summand \(r^j\lambda_j\) accounts for the number and lengths of the words used at level \(j\), for each distinct prime divisor of \(|G|\).

\begin{theorem}[Long cycles in nilpotent Cayley graphs]\label{thm:nilpotent}
Let \(r,c\) be positive integers. Suppose that \(G\) is a finite nilpotent group of order at least three, rank at most \(r\), and class at most \(c\), and that \(S\) is a symmetric generating set of \(G\). Then
\begin{equation}\label{eq:nilpotent}
 \rho\bigl(\Cay(G,S)\bigr)\ge2^{-K_{r,c}\nu(|G|)}.
\end{equation}
Consequently, \(\Cay(G,S)\) contains a cycle and a path, each with at least \(|G|\,2^{-K_{r,c}\nu(|G|)}\) vertices.
\end{theorem}

\begin{proof}[Proof of Theorem~\ref{thm:nilpotent}]
Write \(X=\Cay(G,S)\). We prove~\eqref{eq:nilpotent} by induction on \(c\). If \(G\) is abelian, the cited Hamiltonicity theorem gives \(\rho(X)=1\), proving the bound and, in particular, the base case \(c=1\). If \(G\) has class at most \(c-1\), the induction hypothesis already gives the bound, since \(K_{r,c-1}\le K_{r,c}\). We may therefore assume that \(G\) has class exactly \(c\ge2\).

Put \(H=\gamma_c(G)\le Z(G)\). Since \(\gamma_{c+1}(G)=\{1\}\), Lemma~\ref{lem:shortgenerators} gives generators \(z_1,\ldots,z_t\) of \(H\), with \(t\le r^c\nu(|G|)\), each represented by an \(S\)-word of length at most \(\lambda_c\). Quotient successively by their images, omitting any trivial step. Each kernel is cyclic and central, and its generator is still represented by a word of length at most \(\lambda_c\) in the image of \(S\).

Every intermediate group has at least three elements. Indeed, it has the form \(G/J\) with \(J\le H\le Z(G)\). If \(G/J\) were cyclic, generated by \(gJ\), every element of \(G\) would have the form \(g^a u\) with \(u\in J\). These elements commute because \(J\) is central, contrary to \(G\) being nonabelian. Thus every intermediate group is noncyclic, and Lemma~\ref{lem:shortcentral} applies at every step.

Write \(Q=\Cay(G/H,S_H)\). The group \(G/H\) has rank at most \(r\) and class at most \(c-1\), so the induction hypothesis applies to \(Q\). Combining it with the successive lifting bounds and \(\nu(|G/H|)\le\nu(|G|)\), we obtain
\[
 \rho(X)
 \ge2^{-\lambda_ct}\rho(Q)
 \ge2^{-(r^c\lambda_c+K_{r,c-1})\nu(|G|)}
 =2^{-K_{r,c}\nu(|G|)}.
\]
Multiplying by \(|G|\) gives the cycle bound, and deleting an edge of that cycle gives the path bound.

Finally, the standard estimate \(\nu(m)=O(\log m/\log\log m)=o(\log m)\), which follows from Robin~\cite[Th\'eor\`eme~11]{Robin}, gives
\[
 m\,2^{-K_{r,c}\nu(m)}=m^{1-o(1)}.
\]
Since \(K_{r,c}\) depends only on \(r,c\), this asymptotic bound is uniform over all groups and generating sets in the statement.
\end{proof}

\begin{remark}
The bound in Theorem~\ref{thm:nilpotent} is not optimized. A separate refinement, lifting through the whole centre at each step, gives a cycle containing at least \(\delta_c|G|\) vertices, where \(\delta_c>0\) depends only on \(c\) and no rank bound is needed. We omit the proof of this refinement for clarity of exposition, since the stated bound suffices for Theorem~\ref{thm:main}.
\end{remark}

Cayley graphs of order one or two have a path through every vertex and no ordinary cycle. We handle them directly whenever the path bound is used.

\section{Proof of the main theorem}\label{sec:assembly}
We apply the Tessera--Tointon structure theorem to partition the graph into orbits of small diameter. For small orbits, Babai's contraction lemma reduces the quotient to a nilpotent Cayley graph, to which Theorem~\ref{thm:nilpotent} applies. We then combine this bound with the large-orbit bound from Section~\ref{sec:routes} and a diameter bound for the single-orbit case. Finally, we deduce the cycle bound using Watkins's theorem and the theorem of Bondy and Locke.

\subsection{The structural partition}
We use the following finite consequence of
Tessera and Tointon~\cite[Corollary~2.4]{TT},
with the free-action refinement recorded by
Spanos and Tointon~\cite[Remark~7.3]{SpanosTointon}.

\begin{theorem}[Tessera--Tointon]\label{thm:TT}
For every \(0<\lambda<1\), there exist positive integers \(n_0,r_0,c_0,I_0\), depending only on \(\lambda\), such that every connected vertex-transitive graph \(X\) of order \(n\ge n_0\) has a normal subgroup \(H\normal\Aut(X)\) with the following properties:
\begin{enumerate}[label=\textup{(\roman*)}]
\item every \(H\)-orbit has diameter at most \(n^\lambda\), measured in \(X\);
\item \(Q=X/H\) admits a nilpotent group \(N\le\Aut(Q)\) of rank at most \(r_0\) and class at most \(c_0\), acting freely with at most \(I_0\) orbits on \(V(Q)\).
\end{enumerate}
\end{theorem}

\noindent\emph{Remark.}
Spanos and Tointon~\cite[Remark~7.3]{SpanosTointon}
observe that the nilpotent subgroup in
Tessera and Tointon's finitary structure
theorem~\cite[Theorem~2.3]{TT} can be chosen to act freely,
while retaining the quantitative bounds.
Using this refinement in the proof of
\cite[Corollary~2.4]{TT} gives the free action in~\textup{(ii)}.

\subsection{Babai's contraction lemma}\label{sec:actions}
For \(Q\) and \(N\) supplied by Theorem~\ref{thm:TT}, the free action need not be transitive, so \(Q\) need not be a Cayley graph of \(N\). Babai's lemma below gives a Cayley graph of \(N\) by contracting connected parts of \(Q\). It therefore allows us to apply Theorem~\ref{thm:nilpotent}.

For a graph \(Y\) and a partition \(\mathcal P\) of \(V(Y)\) into nonempty connected sets, write \(Y_{\mathcal P}\) for the graph with vertex set \(\mathcal P\), where distinct parts are adjacent if an edge of \(Y\) joins them. This uses the same adjacency rule as the orbit quotient in Section~\ref{sec:actionnotation}, but the parts need not be subgroup orbits. Replacing each part by a vertex in this way is called \emph{contracting} the parts; edges within a part are discarded, and there is only one edge between each adjacent pair of parts.

We use the following lemma of Babai; see also Seifter and Woess~\cite[Lemma~1]{SW}.

\begin{theorem}[Babai's contraction lemma~\cite{BabaiContraction}]\label{thm:BabaiContraction}
If a finite group \(U\) acts freely by automorphisms on a connected graph \(Y\), then there is a partition \(\mathcal P\) of \(V(Y)\) into nonempty connected sets such that \(Y_{\mathcal P}\) is a connected Cayley graph of \(U\).
\end{theorem}

\begin{remark}[Paths under contraction]\label{rem:contractionpaths}
Let \(\mathcal P\) be a partition of the vertex set of a nonempty graph \(Y\) into nonempty connected sets. Every path \(B_1,\ldots,B_k\) in \(Y_{\mathcal P}\) gives a path in \(Y\) with at least \(k\) vertices. Consequently, \(p(Y)\ge p(Y_{\mathcal P})\).

For \(k=1\), choose any vertex of \(B_1\). For \(k\ge2\), choose an edge \(x_i y_i\) of \(Y\) with \(x_i\in B_i\) and \(y_i\in B_{i+1}\) for each \(1\le i<k\). For each \(2\le i<k\), connectedness of \(B_i\) gives a simple path from \(y_{i-1}\) to \(x_i\) within \(B_i\), using a single vertex if the endpoints coincide. Concatenating these paths with the chosen edges gives a path from \(x_1\) to \(y_{k-1}\). This path is simple because the parts \(B_1,\ldots,B_k\) are distinct and disjoint, and the path used within each part is simple. It visits every \(B_i\), so has at least \(k\) vertices.
\end{remark}

\subsection{Completing the proof}\label{sec:inputs}
We use the following consequence of the expansion bound of Babai and Szegedy~\cite{BabaiSzegedy} and a standard path criterion~\cite[Proposition~7.1]{Krivelevich}; see also~\cite[Lemmas~2.2 and~2.4]{BHMSY}.

\begin{lemma}[A standard diameter bound]\label{lem:diameter}
Every finite connected vertex-transitive graph \(X\) of order \(n\) and diameter \(D\) satisfies
\begin{equation}\label{eq:diameter}
 p(X)\ge\frac{n}{2(D+1)}.
\end{equation}
\end{lemma}
\begin{proof}
The case \(n=1\) is immediate. For \(n\ge2\), put \(k=\lfloor n/2\rfloor\). By Babai and Szegedy's bound, every set of \(k\) vertices has at least \(k/(D+1)\) neighbours outside it. The path criterion~\cite[Proposition~7.1]{Krivelevich} states that, for positive integers \(k,\ell\), a graph with more than \(k\) vertices has a path with at least \(\ell\) edges whenever every set of \(k\) vertices has at least \(\ell\) neighbours outside it. Taking \(\ell=\lceil k/(D+1)\rceil\) therefore gives a path with at least \(\ell\) edges. Consequently,
\[
 p(X)\ge\left\lceil\frac{k}{D+1}\right\rceil+1
 \ge\frac{k+1}{D+1}>\frac{n}{2(D+1)}.\qedhere
\]
\end{proof}

\begin{proof}[Proof of Theorem~\ref{thm:main}]
It suffices to consider \(0<\eps<1\). The parameter \(\lambda\) controls the orbit diameters, while \(\eta\) is the exponent loss in the Cayley-graph bound. Set \(\lambda=\eps/8\) and \(\eta=\eps/4\), and let \(r,c,I\) be the bounds supplied by Theorem~\ref{thm:TT}. All thresholds below depend only on \(\eps\).

Let \(X\) be a connected vertex-transitive graph of sufficiently large order \(n\). We first prove
\begin{equation}\label{eq:mainpathbound}
 p(X)\ge n^{1-\eps}.
\end{equation}
Apply Theorem~\ref{thm:TT} to obtain \(H\normal\Aut(X)\) and \(Q=X/H\). Write \(q=|V(Q)|\) and \(s=n/q\) for the common orbit size, which is well-defined by Lemma~\ref{lem:orbits}, and put
\[
 R=n^\lambda+1\le2n^\lambda.
\]
Each \(H\)-orbit has diameter at most \(R-1\) in \(X\).

\smallskip
\noindent\emph{A single orbit: \(q=1\).}
Here \(\diam(X)\le n^\lambda\), so Lemma~\ref{lem:diameter} gives
\[
 p(X)\ge\frac{n}{2(n^\lambda+1)}
 \ge\frac{n^{1-\eps/8}}4\ge n^{1-\eps}
\]
for sufficiently large \(n\).

\smallskip
\noindent\emph{Large orbits: \(q\ge2\) and \(s\ge n^{\eps/2}\).}
For sufficiently large \(n\),
\[
 s\ge n^{\eps/2}\ge768n^{\eps/4}\ge192R^2.
\]
Lemma~\ref{lem:largefibres} therefore gives
\[
 p(X)\ge\frac{n}{1000R^2}
 \ge\frac{n^{1-\eps/4}}{4000}\ge n^{1-\eps}.
\]

\smallskip
\noindent\emph{Small orbits: \(q\ge2\) and \(s<n^{\eps/2}\).}
Now \(q>n^{1-\eps/2}\). Let \(N\) be the nilpotent subgroup supplied by Theorem~\ref{thm:TT}, and put \(m=|N|\). The group \(N\) has at most \(I\) orbits on \(Q\), each of size \(m\), by freeness. Thus \(m\ge q/I\).

Theorem~\ref{thm:BabaiContraction} gives a connected Cayley graph \(Y\) of \(N\) by contracting connected parts of \(Q\). Theorem~\ref{thm:nilpotent} gives \(p(Y)\ge m^{1-\eta}\) for sufficiently large \(n\), since \(m\ge q/I\to\infty\). Remark~\ref{rem:contractionpaths} and the orbit lifting in Lemma~\ref{lem:orbits} now yield
\[
 p(X)\ge p(Q)\ge p(Y)\ge m^{1-\eta}
 \ge\frac{n^{(1-\eps/2)(1-\eps/4)}}{I^{1-\eps/4}}
 \ge n^{1-\eps}.
\]
The last inequality holds for sufficiently large \(n\), because \(I\) depends only on \(\eps\) and \((1-\eps/2)(1-\eps/4)>1-\eps\). This proves~\eqref{eq:mainpathbound}.

For cycles, apply~\eqref{eq:mainpathbound} with \(\eps/2\) in place of \(\eps\). For \(n\ge4\), a connected vertex-transitive graph has degree at least two. If its degree is two, it is a cycle. Otherwise, Watkins's theorem~\cite{Watkins} makes it 3-connected, and the theorem of Bondy and Locke~\cite{BL} gives
\[
 c(X)\ge\frac25\bigl(p(X)-1\bigr)
 \ge\frac25\bigl(n^{1-\eps/2}-1\bigr)
 \ge\frac15n^{1-\eps/2}\ge n^{1-\eps}
\]
for sufficiently large \(n\).
\end{proof}

\section*{Statement of AI use}
The starting point of this work was the expander decomposition lemma of Letzter, Methuku and Sudakov~\cite[Lemma~4.1]{LMS}, which partitions almost all vertices of the host graph into parts with useful regularity and expansion properties. It does not, however, ensure that these parts are vertex-transitive. In 2025, assuming a partition into sufficiently large vertex-transitive parts, the authors developed the idea of using successive traversals of a spanning tree in the quotient, each traversal using every edge once in each direction, together with the Lov\'asz local lemma to control intersections between random short paths and extract a long simple path. When the parts were small, however, the authors' ideas for lifting paths incurred substantial losses.

In July 2026, when the authors asked OpenAI's GPT-5.6 Sol whether a decomposition into vertex-transitive parts is possible, it directed them to the structure theorem of Tessera and Tointon~\cite{TT}. This supplies precisely such a partition, with several additional useful properties. Combining this partition with a diameter bound and their spanning-tree and local-lemma argument, the authors quickly resolved the cases of a single orbit and of large orbits, described in Sections~\ref{sec:overviewsingle} and~\ref{sec:overviewroutes} of the overview.

For the small-orbit case, after the authors presented their earlier ideas for lifting paths and the loss in that argument, GPT pointed them to Babai's contraction lemma~\cite{BabaiContraction}. This reduced the problem to finding a long path in a connected Cayley graph of a nilpotent group with boundedly many generators and bounded nilpotency class. With the help of GPT, the authors combined their earlier ideas for lifting long paths with an argument used for proving the classical Factor Group Lemma to obtain the bound in Lemma~\ref{lem:shortcentral}, showing that lifting through a cyclic central quotient retains at least \(1/2^\ell\) of the proportion of vertices covered in the quotient, where the central generator is a product of at most \(\ell\) prescribed generators. However, the authors did not know how to ensure that \(\ell\) is sufficiently small. GPT supplied the required group-theoretic arguments for controlling both the number and the word lengths of the required generators (Lemma~\ref{lem:shortgenerators}), keeping the total loss to \(m^{o(1)}\), where \(m\) is the group order. This provided the key remaining step for completing the small-orbit case. The authors also used AI assistance to polish and rewrite parts of this manuscript and take full responsibility for the final content.

\section*{Acknowledgements}
The authors are grateful to Matija Buci\'c, Shoham Letzter and Alp M\"uyesser for several insightful discussions on the topic of this paper.

\end{document}